\documentclass[12pt,english]{amsart}

\usepackage[inner=1in,outer=1in,top=1in,bottom=1in]{geometry}
\usepackage{amssymb}
\usepackage{paralist}
\usepackage{babel}
\usepackage[pdftex]{graphicx}    
\usepackage[leqno]{amsmath}
\usepackage{amsthm}
\usepackage{emptypage}
\usepackage{hyperref}

\usepackage{tikz-cd}

\usepackage[utf8]{inputenc}
\usepackage{enumerate}
\usepackage{amsfonts,mathrsfs}
\usepackage{ulem}
\usepackage{paralist}
\usepackage[leqno]{amsmath}
\usepackage{graphics}
\usepackage[absolute]{textpos}

\newcommand\mbb{\mathbb}

\newcommand\ul{\underline}

\newcommand\C{\mbb{C}}

\newcommand\R{\mbb{R}}
\newcommand\Z{\mbb{Z}}

\DeclareMathOperator*{\x}{x}

\DeclareMathOperator*{\vt}{t}

\DeclareMathOperator*{\adj}{adj}

\DeclareMathOperator*\Hom{Hom}

\DeclareMathOperator*{\Sym}{Sym}

\DeclareMathOperator*{\Co}{C}

\newcommand\tensor{\otimes}

\renewcommand\epsilon{\varepsilon}

\renewcommand\mod{\mathop{\rm mod}}
\renewcommand\phi{\varphi}

\renewcommand\theta{\vartheta}
\theoremstyle{plain}
\newtheorem{Thm}{Theorem}
\newtheorem{Prop}[Thm]{Proposition}
\newtheorem{Cor}[Thm]{Corollary}
\newtheorem{Lemma}[Thm]{Lemma}

\newtheorem*{Thm*}{Theorem}
\newtheorem*{Prop*}{Proposition}
\newtheorem*{Cor*}{Corollary}
\newtheorem*{Lemma*}{Lemma}
\newtheorem*{Sublemma*}{Sublemma}
\newtheorem*{Conjecture*}{Conjecture}
\theoremstyle{definition}

\newtheorem{Def}[Thm]{Definition}

\newtheorem{Example}[Thm]{Example}

\newtheorem{Remark}[Thm]{Remark}

\newtheorem*{Constr*}{Construction}
\newtheorem*{Def*}{Definition}
\newtheorem*{Defs*}{Definitions}
\newtheorem*{Example*}{Example}
\newtheorem*{Examples*}{Examples}
\newtheorem*{Exercise*}{Exercise}
\newtheorem*{LemmaDef*}{Lemma and Definition}
\newtheorem*{Notation*}{Notation}
\newtheorem*{Problem*}{Problem}
\newtheorem*{Question*}{Question}
\newtheorem*{Remark*}{Remark}
\newtheorem*{Remarks*}{Remarks}
\newtheorem*{Warning*}{Warning}

\newtheorem*{Text*}{}
\numberwithin{equation}{section}
\numberwithin{Thm}{section}

\begin{document}

\title{Determinantal Representations and B\'ezoutians}

 \author{Mario Kummer}
 \address{Universit\"at Konstanz, Germany} 
 \email{Mario.Kummer@uni-konstanz.de}

\begin{abstract}
 A major open question in convex algebraic geometry is whether all hyper\-bolicity cones are spectrahedral, i.e. the solution sets of linear matrix inequalities.
 We will use sum-of-squares decompositions of certain bilinear forms
  called B\'ezoutians to approach this problem. More precisely, 
 we show that for every smooth hyperbolic polynomial $h$ there is another hyperbolic polynomial $q$ such that $q \cdot h$ has a definite
 determinantal representation. 
 Besides commutative algebra, the proof relies on results from real algebraic geometry.
\end{abstract}

\maketitle

\section{Introduction}
A homogeneous polynomial $h \in \R[x_0,\ldots,x_n]$ is said to be \textit{hyperbolic with respect to}
$e \in \R^{n+1}$, if $h$ does not vanish in $e$ and if for every 
$v \in \R^{n+1}$, the univariate polynomial $h(t e +v) \in \R[t]$ has only real roots.
The \textit{hyperbolicity cone} $\Co _h (e)$ of $h$ at $e$ is the set of all $v \in \R^{n+1}$ such that no zero of 
$h(t e +v)$ is strictly positive.
Hyperbolicity cones are semi-algebraic convex cones, as shown for example in \cite{Gar}.

The interest in hyperbolic polynomials was originally motivated by the theory of partial differential equations (see for example 
\cite{Gar1, Lax}). But  lately, interest arose in the area of optimization, especially semidefinite optimization (see for example 
\cite{Gue, HV, Ren}). In particular the connection to polynomials with a \textit{definite determinantal representation} has attracted much
attention: A homogeneous polynomial $h \in \R[x_0, \ldots, x_n]$ has a definite determinantal representation, if there are symmetric matrices
$A_0, \ldots, A_n \in \Sym_d(\R)$ with \[h= \det(x_0\cdot A_0+\ldots+x_n\cdot A_n)\] and if $A(e)=e_0A_0+\ldots+e_nA_n$ is
positive definite for some 
$e \in \R^{n+1}$. 
Note that both concepts, that of hyperbolicity as well as that of a definite determinantal representation,
have been generalized to varieties of higher codimension in \cite{SV}.
It is easy to see that polynomials with a determinantal representation which
is definite at $e$ are hyperbolic with respect to $e$. An important result of Helton and Vinnikov
\cite{HV} says that conversely every hyperbolic polynomial in three variables has a definite determinantal representation.
This holds no longer true for more than three variables.
Actually Br\"and\'en \cite{Bra11} found a hyperbolic polynomial $h$ in four variables
such that no power $h^N$ admits a definite determinantal representation.
More about these topics can be found in \cite{vamosext,vamosspec,interlacing,hermitehv, Vin}.

A \textit{spectrahedral cone} is a set 
defined by some homogeneous linear matrix inequalities, i.e. sets of the form 
\[\{v \in \R^{n+1}: \,\, A(v)=v_0 A_0 + \ldots + v_n A_n \succeq 0\},\]
where $A_0, \ldots, A_n$ are symmetric matrices with real entries.
Spectrahedral cones are of interest since they are the
feasible sets of semi-definite programming.
One of the major open problems in convex algebraic geometry is to find a characterization of the class of
spectrahedral cones.
It is not hard to check that every
spectrahedral cone is the hyperbolicity cone of an appropriate hyperbolic polynomial.
Note that hyperbolicity of a cone is usually easier to check than spectrahedrality:
One has to determine the Zariski closure of the boundary of the cone and check whether the defining
polynomial is hyperbolic. There is no such criterion known for spectrahedral cones.
The \textit{Generalized Lax Conjecture} states that in fact every hyperbolicity cone
is a spectrahedral cone.
\begin{Conjecture*}
Every hyperbolicity cone is a spectrahedral cone.
\end{Conjecture*}
Here are some selected results concerning this conjecture:
\begin{itemize}
 \item It is true for hyperbolic polynomials in at most three variables. This is a consequence of the above mentioned theorem of 
 Helton and Vinnikov \cite{HV}.
 \item It is true for the hyperbolicity cones of elementary symmetric polynomials \cite{Bra14}.
 \item Every smooth hyperbolicity cone is the \textit{projection} of a spectrahedral cone \cite{NS}.
 \item It is true if $-1$ is not a sum of squares in certain non-commutative algebras \cite{NT}.
\end{itemize}
It is well known that the Generalized Lax Conjecture is equivalent to the following:
\begin{Conjecture*}
 Let $h \in \R[x_0,\ldots,x_n]$ be hyperbolic with respect to $e \in \R^{n+1}$.
 Then 
 there is a hyperbolic polynomial $q \in \R[x_0,\ldots,x_n]$, such
 that the following two conditions are
 satisfied:
 \begin{enumerate}[(i)]
  \item The product $q \cdot h$ has a definite determinantal representation.
  \item We have $\Co_h(e) \subseteq \Co_q(e)$.
 \end{enumerate}
\end{Conjecture*}

Our contribution is to show that if the hyperbolic polynomial has no real singularities (which is generically the case),
then there will always be such a hyperbolic polynomial $q$ that satisfies the first of these two conditions (Theorem \ref{thm:main}).

\begin{Thm*}
 Let $h \in \R[x_0, \ldots, x_n]$
 be hyperbolic with respect to $e \in \R^{n+1}$. Assume that $h$ has no real singularities (i.e. $\nabla h (v) \neq 0$ for all
 $0 \neq v \in \R^{n+1}$). Then there is a hyperbolic polynomial 
 $q \in \R[x_0, \ldots, x_n]$  such that $q\cdot h$ has a definite determinantal representation.
\end{Thm*}
Note that if one omits the condition that the determinantal representation is definite at some point,
a stronger statement is true:
If $f \in \R[x_0, \ldots, x_n]$ is a homogeneous polynomial, then there is a linear form $l\in \R[x_0, \ldots, x_n]$,
such that $l^N \cdot f$ has a
(not necessarily definite) determinantal representation, see \cite{HMcV}.

This article is organized as follows.
After some preliminaries in Section \ref{sec:prelim} we will recall in Section \ref{sec:bezintro} 
the notion of a B\'ezoutian and remind the reader of some basic properties. Consider 
a homogeneous polynomial $h \in S= \R[x_0, \ldots, x_n]$
which does not vanish in the point $e=(1,0,\ldots,0)^{\textrm{T}}$. Then the module $M=S/(h)$ is free over the
ring $R=\R[x_1, \ldots, x_n]$. We will call a symmetric bilinear form on its dual $\Hom_R(M,R)$ a B\'ezoutian, if 
the corresponding $R$-linear map $\Hom_R(M,R) \to M$ is even $S$-linear. The precise definition is given in \ref{def:bez}.
If $h$ is a smooth polynomial, then it is hyperbolic with respect to $e=(1,0,\ldots,0)^{\textrm{T}}$
if and only if there is a positive definite B\'ezoutian.
This follows from classic properties of the B\'ezout matrix which we will recall 
in Section \ref{sec:bezhyp}. If there
is a B\'ezoutian which has a sum of squares decomposition of some nice form, we can construct from that a definite determinantal
representation of a multiple of $h$, see Section \ref{sec:bezdet2}.
Note that a similar approach was tried in \cite{NPT}, using the Hermite matrix instead of the B\'ezout matrix.
They obtained a determinantal representation with \textit{rational functions} as entries.
In Section \ref{sec:stellensatz} we will prove
that if $h$ is a smooth hyperbolic polynomial, then such a B\'ezoutian always exists, which then implies
Theorem \ref{thm:main}.
The main ingredients for that proof are Stengle's Positivstellens\"atze.

\section{Preliminaries and Notation}{\label{sec:prelim}}
Let $R=\R[x_1,\ldots, x_n]$ be the polynomial ring equipped with the natural
grading. If $M$ is a graded $R$-module, we denote by $M_d$ the set of homogeneous elements in $M$ of degree $d$.
A \textit{free graded $R$-module} is a finitely generated graded $R$-module of the form
\[M=R(e_1) \oplus R(e_2) \oplus \ldots \oplus R(e_r)\] where $e_1, \ldots, e_r$ are integers and $R(e_i)$ is the twist of $R$ by $e_i$.
Let $M$ and $N$ always denote free graded $R$-modules.
We denote the \textit{dual module} of $M$ by $M^* := \Hom_R(M,R)$. 
We consider on the $R$-module $L=\Hom_R(M,N)$ the induced grading given by
\[ L_d=\{\varphi \in L: \,\, \varphi(M_e) \subseteq M_{d+e} \textrm{ for all } e \in \Z\}. \]
In particular, we can think of $M^*$ as a free graded $R$-module.
 We will often use the natural isomorphism of $R$-modules $M \otimes_R M \cong \Hom_R(M^*,M)$
which sends an elementary tensor $v \otimes w \in M \otimes_R M$ to the homomorphism $\varphi \mapsto \varphi(v) \cdot w$.

\begin{Remark}
 This isomorphism induces a grading on $M\otimes_R M$: If $v \in M_d$ and $w \in M_e$, then $v \otimes w$ is homogeneous of degree $e+d$.
\end{Remark}

\begin{Remark}
 If $f \in \Hom_R( M , N)$ and $\alpha \in \Hom_R(M^*,M)$, then
 $f \circ \alpha \circ f^* \in \Hom_R(N^*,N)$, where $f^*\in \Hom_R( N^* , M^*)$ is the dual map.
 If we think of $\alpha$ as an element of $M\otimes_R M$ as in the identification above, then 
 $f \circ \alpha \circ f^*$ corresponds to $(f \otimes f)(\alpha) \in N\otimes_R N$.
\end{Remark}

To every element $\alpha \in \Hom_R(M^*,M)$ we associate the $R$-bilinear form on $M^*$ given by
$\langle \varphi, \psi \rangle_{\alpha} = \psi(\alpha(\varphi))$. We say that $\alpha \in \Hom_R(M^*,M)$ is (skew-)symmetric
if the associated bilinear form is (skew-)symmetric.
 This identification associates $v \otimes w \in M\otimes_R M$ to the bilinear form $(\varphi, \psi)\mapsto \varphi(v) \cdot \psi(w)$.

Let $\alpha \in \Hom_R(M^*,M)$. Using biduality, we can think of the dual map $\alpha^*$ again as an element of $\Hom_R(M^*,M)$.
The associated bilinear form satisfies $\langle \varphi, \psi \rangle_{\alpha^*}=\langle \psi, \varphi \rangle_{\alpha}$ for all
$\varphi, \psi \in M^*$. Thus $\alpha$ is (skew-)symmetric if $\alpha=\alpha^*$ (resp. $\alpha=-\alpha^*$).

Now we consider the $R$-linear map \[\Hom\nolimits_R(M^*,M) \to \Hom\nolimits_R(M^*,M), \,\, f \mapsto \frac{1}{2} (f-f^*).\]
Its image is exactly the set of skew-symmetric $R$-bilinear forms on $M^*$. Under the identification 
$M \otimes_R M \cong \Hom_R(M^*,M)$ its kernel is generated by all the elements $v \otimes v$ where $v \in M$.
This yields a natural identification between the set of skew-symmetric $R$-bilinear forms on $M^*$ and the second
exterior power $\bigwedge^2 M$.

\begin{Remark}{\label{rem:skewhom}}
 If $f \in \Hom_R( M , N)$ and $\alpha \in \Hom_R(M^*,M)$ is skew-symmetric, then
 the map $f \circ \alpha \circ f^* \in \Hom_R(N^*,N)$ is again skew-symmetric.
 If we think of $\alpha$ as an element of $\bigwedge^2 M$ as in the identification above, then 
 $f \circ \alpha \circ f^*$ corresponds to $\wedge f (\alpha) \in \bigwedge^2 N$, where
 $\wedge f: \bigwedge^2 M \to \bigwedge^2 N$ is the map
 induced by $f$.
\end{Remark}

\begin{Remark}{\label{rem:skewkern}}
 Let $f: M \to N$ be an $R$-linear map.
 Let $L \subseteq \bigwedge^2 M$ be the submodule generated by all elements
 $v \wedge w$ where $v \in \ker(f)$ and $w \in M$. Each element of $L$ corresponds to a skew-symmetric $\alpha \in \Hom_R(M^*,M)$
 which can be written in the form $\alpha=\beta-\beta^*$ where $\textnormal{im}(f^*) \subseteq \ker(\beta)$: Indeed, $v \wedge w$ corresponds to the map
 $\varphi \mapsto \frac{1}{2}(\varphi(v) w -\varphi(w) v)$ and the map $\varphi \mapsto \frac{1}{2} \varphi(v) w$ vanishes on $\textnormal{im}{f^*}$
 whenever $v \in \ker(f)$.
\end{Remark}

\section{B\'ezoutians}{\label{sec:bezintro}}
In this section we will recall some folklore about what is called B\'ezoutians. 

We consider the polynomial ring $S=\R[x_0,\ldots, x_n]$ which is
equipped with the natural grading $S=\bigoplus_{i=0}^{\infty} S_i$
where every variable has degree $1$. For the entire section we fix a homogeneous polynomial
$h \in S_d$ of degree $d>0$ such that $h(e)\neq 0$ where $e=(1,0,\ldots,0)$.

We also consider the graded ring $R=\R[x_1, \ldots, x_n]$ and
its natural embedding $R \hookrightarrow S$. Every graded $S$-module can also be considered as a graded $R$-module.

Since $h(e) \neq 0$, the graded $R$-module $M:=S/(h)$  is  free of rank $d$.
Let $M^*=\Hom_{R}(M,R)$ be its dual.
We can consider $M^*$ as an $S$-module by letting 
$(a \cdot \varphi)(x) := \varphi(a \cdot x)$ for all $\varphi \in M^*$,
$a \in S$ and $x \in M$. We also give $M^*$ the structure of a graded
$S$-module in the natural way.

\begin{Def}{\label{def:bez}}
 Let $\langle-,- \rangle$ be a symmetric $R$-bilinear form on $M^*$.
 We say that $\langle-,- \rangle$ is a \textit{B\'ezoutian} (with respect to $e$) if it satisfies
 $\langle a \cdot \varphi, \psi \rangle=\langle  \varphi,a \cdot \psi \rangle$ for all $a \in S$,
 $\varphi, \psi \in M^*$.
\end{Def}

Since $M$ is a free $R$-module, we can identify the set of $R$-bilinear forms on $M^*$
with $M \otimes_{R} M$ resp. $\Hom_R(M^*,M)$ as in the previous section. 
The B\'ezoutians are exactly those symmetric elements of $\Hom_R(M^*,M)$ that are actually $S$-linear.
Furthermore we have a natural isomorphism of $R$-algebras
\[M \otimes_R M \cong R[s,t]/(h(s),h(t))\]
where $h(s)=h(s,x_1, \ldots, x_n)$ resp. $h(t)=h(t,x_1, \ldots, x_n)$.
This isomorphism identifies $s$ with $\overline{x}_0 \otimes 1$ and $t$ with $1 \otimes \overline{x}_0$.
We say that \[\omega \in M \otimes_{R} M \cong \Hom\nolimits_R(M^*,M) \cong R[s,t]/(h(s),h(t)) \] is a B\'ezoutian
if the corresponding bilinear form, which
we denote by $\langle-,-\rangle_{\omega}$, is one.

\begin{Example}{\label{exp:bez}}
 Let $p(s,t) \in R[s,t]$ be homogeneous and let $\omega$ be the residue class of the polynomial \[\frac{h(s) p(s,t)-h(t) p(t,s)}{s-t}\]
 in $R[s,t]/(h(s),h(t))$.
 Then $\omega$ is a B\'ezoutian.
 Indeed, it is clear that $\omega$ is symmetric. And since we have 
 \[s \cdot \frac{h(s)p(s,t)-h(t)p(t,s)}{s-t} \equiv t \cdot \frac{h(s)p(s,t)-h(t)p(t,s)}{s-t} \mod (h(s),
 h(t))\]
 it follows that $\langle x_0\cdot \varphi, \psi \rangle_{\omega}=
 \langle  \varphi,x_0\cdot \psi \rangle_{\omega}$
 for all  $\varphi, \psi \in M^*$ which implies that $\omega$ is a B\'ezoutian.
 The B\'ezoutian obtained from the polynomial $\frac{h(s)-h(t) }{s-t}$ will be denoted by $\delta_h$.
\end{Example}

\begin{Lemma}{\label{lem:faktorbez}}
 Let $p(s,t) \in R[s,t]$. Then the polynomial $\frac{h(s) p(s,t)-h(t) p(t,s)}{s-t} \in R[s,t]$ is equivalent to
 $\frac{h(s) -h(t) }{s-t} \cdot p(s,t)$ modulo the ideal $(h(t)) \subseteq R[s,t]$.
\end{Lemma}

\begin{proof}
 We can write \begin{eqnarray*}
               h(s)&=&h(t)+(s-t)\cdot r_h \\
               p(s,t)&=&p(t,t)+(s-t)\cdot r_{p,1} \\
               p(t,s)&=&p(t,t)+(s-t)\cdot r_{p,2} 
              \end{eqnarray*}
where $r_h, r_{p,1}, r_{p,2} \in R[s,t]$. Then we have
\begin{eqnarray*}
 \frac{h(s) p(s,t)-h(t) p(t,s)}{s-t}& =&h(t)\cdot r_{p,1}+r_h\cdot p(t,t)+(s-t)\cdot r_h\cdot r_{p,1}
 -h(t) \cdot  r_{p,2} \\
 &=& r_h \cdot p(s,t) + h(t) \cdot (r_{p,1}-r_{p,2}   )   .                                             
\end{eqnarray*}
\end{proof}

\begin{Prop}{\label{prop:bezideal}}
 The set of B\'ezoutians in $M \otimes_R M$ is exactly the principal ideal in $M \otimes_R M$ generated by $\delta_h$.
\end{Prop}

\begin{proof}
 Let $f(s,t) \in R[s,t]$ be a polynomial such that its residue class in $ R[s,t]/(h(s),h(t))\cong M \otimes_R M $
 is a B\'ezoutian. This means that 
 \[ f(s,t)-f(t,s) \in (h(s),h(t)) \textrm{ and }
 (s-t) \cdot f(s,t) \in (h(s),h(t)).\]
 Since we have $f(s,t) \equiv \frac{1}{2} (f(s,t)+f(t,s)) \mod (h(s),h(t))$ we can assume without loss of
 generality that $f(s,t)=f(t,s)$. There are polynomials $p(s,t),q(s,t) \in R[s,t]$ such that
 \[(s-t) f(s,t)=p(s,t) h(s)+q(s,t) h(t).\]
 Using the fact that $f(s,t)=f(t,s)$ we obtain
 \[(s-t) f(s,t) = \frac{1}{2} (p(s,t)-q(t,s)) h(s) - \frac{1}{2} (p(t,s)-q(s,t)) h(t).\]
 Thus the preceding Lemma implies that
 \[f(s,t) \equiv \frac{1}{2} (p(s,t)-q(t,s)) \cdot \frac{h(s) -h(t) }{s-t} \mod (h(s),h(t)).\]
 Therefore every B\'ezoutian is a multiple of $\delta_h$.
 
 Conversely, Lemma \ref{lem:faktorbez} together with Example \ref{exp:bez} shows that every multiple of $\delta_h$ is
 in fact a B\'ezoutian.
\end{proof}

\begin{Remark}
Consider the homomorphisms of $R$-algebras \[\lambda_1, \lambda_2: M \to M \otimes_R M\] defined by
$\lambda_1(z)=z \otimes 1$ and $\lambda_2(z)=1 \otimes z$.
Since $\delta_h$ is a B\'ezoutian, we have 
\[(x \tensor y) \cdot \delta_h = (xy \tensor 1) \cdot \delta_h =(1 \tensor xy) \cdot \delta_h\]
for all $x,y \in M$.
Therefore we have a homomorphism of $R$-modules
\[M \to M \otimes_R M, \,\, z \mapsto \lambda_1(z) \cdot \delta_h= \lambda_2(z) \cdot \delta_h\] 
whose image is the set of B\'ezoutians.
\end{Remark}

\begin{Def}
 Let $p \in S$ and $\overline{p} \in M$ its residue class.
 We say that $\lambda_1(\overline{p})\cdot \delta_h=\lambda_2(\overline{p}) \cdot \delta_h$ is the
 B\'ezoutian corresponding to $p$.
\end{Def}

\section{B\'ezoutians and Hyperbolicity}{\label{sec:bezhyp}}
Let $N$ be a free module over $R=\R[x_1, \ldots, x_n]$ and let $N^{*}$ be its dual.
For every point $p \in \R^n$ let $\mathfrak{m}_p \subseteq R$ be the corresponding maximal ideal.
The natural evaluation map 
\[\epsilon_p: N \otimes_R N \to (N \otimes_R (R/ \mathfrak{m}_p)) \otimes_{R/ \mathfrak{m}_p} 
 (N \otimes_R (R/ \mathfrak{m}_p))
\]
sends every symmetric $R$-bilinear form on $N^{*}$ to a symmetric $\R$-bilinear form on the $\R$-vector space
$(N \otimes_R (R/ \mathfrak{m}_p))^{*}$. 
\begin{Def}
 Let $\omega \in N \otimes_R N$ be a homogeneous, symmetric bilinear form on $N^{*}$.
 We say that $\omega$ is \textit{positive (semi)-definite} if the $\R$-bilinear form $\epsilon_p(\omega)$
 is positive (semi)-definite for every $0 \neq p \in \R^n$.
\end{Def}

\begin{Remark}{\label{rem:repmatr}}
 Let $v_1, \ldots, v_m$ be a basis of $N$ and let $v_1^*, \ldots, v_m^* \in N^{*}$  be its dual basis.
 Let \[\omega= \sum_{i,j=1}^m \lambda_{ij} \cdot v_i \otimes v_j \in N \otimes_R N\]
 for some $\lambda_{ij} \in R$. The representing matrix of $\omega$ as a bilinear form on $N^{*}$
 with respect to $v_1^*, \ldots, v_m^*$ is given by $\Lambda=(\lambda_{ij})_{1 \leq i,j \leq m}$.
 A representing matrix of $\epsilon_p(\omega)$ for $p \in \R^n$ is given by 
 $\Lambda(p):=(\lambda_{ij}(p))_{1 \leq i,j \leq m}$ where  $\lambda_{ij}(p) \in \R$ is just
 the value of the polynomial $\lambda_{ij}$ at $p$.
\end{Remark}

Let $f, g \in \R[t]$ be two univariate polynomials having degrees $\deg(f)=d$ and $\deg(g) < d$.
The \textit{B\'ezout matrix} of $f$ and $g$ is defined as follows. We have
\[\frac{f(s)g(t)-f(t)g(s)}{s-t}=\sum_{i,j=1}^d b_{ij} s^{i-1} t^{j-1} \]
for some real numbers $b_{ij}$. Then the B\'ezout matrix $\textnormal{B}(f,g)=(b_{ij})_{ij}$ is a real symmetric 
matrix. The next Theorem states a classically known property of the B\'ezout matrix
which will be very important for what follows.

\begin{Thm}[see §2.2 of \cite{krein}]{\label{thm:bezuni}}
 Let $f \in \R[t]$ be an univariate polynomial. Then the following are equivalent:
 \begin{enumerate}[(i)]
  \item Every zero of $f$ is simple and real.
  \item There is a $g \in \R[t]$ of degree $\deg(g) < \deg(f)$ such that the B\'ezout matrix
  $B(f,g)$ is positive definite.
  \item The B\'ezout matrix $B(f,f')$ is positive definite where $f'$ is the derivative of $f$.
 \end{enumerate}
\end{Thm}

As in the previous section we let $h$ be an homogeneous polynomial in $S=\R[x_0,\ldots,x_n]$ of degree $d$ which
does not vanish in $e=(1,0,\ldots,0)$. 
Theorem \ref{thm:bezuni} gives us the following connection between hyperbolicity of $h$ and B\'ezoutians in the sense of
Definition \ref{def:bez}:

\begin{Thm}{\label{thm:bezhyp}}
 Let $h$ be a homogeneous polynomial in $S=\R[x_0,\ldots,x_n]$ of degree $d$ which
 does not vanish in $e=(1,0,\ldots,0)$.
 Assume that $h$ has no real singularities (i.e. $\nabla h (v) \neq 0$ for all $0 \neq v \in \R^{n+1}$).
 Then the following
 are equivalent:
 \begin{enumerate}[(i)]
  \item The polynomial $h$ is hyperbolic with respect to $e$.
  \item There is a homogeneous B\'ezoutian in $S/(h) \otimes_R S/(h)$ which is positive definite.  
 \end{enumerate}
\end{Thm}

\begin{proof}
 $(i) \Rightarrow (ii)$:  \cite[Theorem 5.2]{HV} states that if $h$ is hyperbolic with respect to $e$,
and if
$h$ has no real singularities,
then the polynomial $h(t,v) \in \R[t]$ has only simple and real roots for all
$0 \neq v \in \R^n$. By Theorem \ref{thm:bezuni} the B\'ezout matrix
$\textrm{B}(h(t,v),\frac{\partial h }{\partial x_0}(t,v))$ is positive definite for all $0 \neq v \in \R^n$
i.e. the B\'ezoutian corresponding to $\frac{\partial h}{\partial x_0}$ is positive definite.

$(ii) \Rightarrow (i):$ Let $\omega \in S/(h) \otimes_R S/(h)$ such a B\'ezoutian.
We have seen in the previous section that $\omega$ is the B\'ezoutian corresponding to some polynomial $p \in S$.
In particular we have seen that $\omega$ is the residue class of the polynomial
\[\frac{h(s) p(t)-h(t) p(s)}{s-t}\, \mod (h(s),
 h(t))\] where $h(s)=h(s,x_1,\ldots,x_n)$ and $h(t), p(s)$ and $p(t)$ are
defined analogously. Moreover, we can assume that the degree of $p$ in $x_0$ is smaller than $d=\deg(h)$.
Now let $0 \neq v \in \R^n$. If we look at the representing matrix of $\omega$ with respect
to the basis $1,\overline{x}_0, \ldots, \overline{x}_0^{d-1}$ as in Remark \ref{rem:repmatr}, we see that $\epsilon_v(\omega)$ has
$\textrm{B}(h(t,v),p(t,v))$ as a representing matrix. This is by assumption positive definite, thus by Theorem
\ref{thm:bezuni} the univariate polynomial $h(t,v)$ has only real roots for all $v \in \R^{n+1}$. Therefore $h$
is hyperbolic with respect to $e$.
\end{proof}

\section{B\'ezoutians and Determinantal Representations}{\label{sec:bezdet2}}
Let $h \in S= R[x_0]=\R[x_0, \ldots, x_n]$ be homogeneous of degree $d$ and irreducible such that
$h$ is hyperbolic with respect to
$e=(1,0,\ldots,0)$.
We are interested in finding symmetric matrices $A_1, \ldots, A_n \in \Sym_N (\R)$ of size $N$ with real entries,
such that $h$ divides \[\det(x_0 \textrm{I}_N -( x_1 A_1+ \ldots + x_n A_n))\] where $\textrm{I}_N$ is the identity matrix.

Let $M=R[x_0]/(h)$ be equipped with the grading induced by the natural grading of $R[x_0]$.
As we have already mentioned, $M$ is a free graded $R$-module. 

\begin{Def}
 Let $\omega \in M \otimes_R M$ be a homogeneous B\'ezoutian.
 We say that $\omega$ is a \textit{nice B\'ezoutian}, if it can be written as
 \[\omega=v_1 \otimes v_1 + \ldots + v_r \otimes v_r\] for some $v_1, \ldots, v_r \in M_e$ 
 which generate the $R$-module $M_{\geq e}$.
\end{Def}

We will show that we can always construct from a nice B\'ezoutian $\omega \in S\otimes_R S$, if there exists one,
a determinantal representation of some multiple of $h$ as above. They key will be the following rather technical lemma.

\begin{Lemma}
 Let $\mathfrak{m}=(x_1, \ldots, x_n)$ (ideal of $R$).
 Let $i: M_{\geq e} \hookrightarrow M$ be the inclusion map.
 The induced map $\wedge i: \bigwedge^2 M_{\geq e} \to \bigwedge^2 M$ (over $R$)
 is injective on $\mathfrak{m}(\bigwedge^2 M_{\geq e})$.
\end{Lemma}

\begin{proof}
 Let $\underline{\x}=(x_0, \ldots, x_n)$. 
 For $\alpha=(\alpha_0,\ldots,\alpha_n) \in \Z^{n+1}$ we write $|\alpha|:=\sum_{i=0}^n |\alpha_i|$ and
$\underline{\x}^{\alpha}:=x_0^{\alpha_0} \cdots x_n^{\alpha_n}$.
 First we consider the set \[T=\{(\alpha,\beta): \,\, \alpha, \beta \in \Z_{\geq 0}^{n+1},\,
 |\alpha|\geq e+1,\, |\beta| = e, \, 0 \leq \alpha_0 < \beta_0 \leq d-1 \}.\]
 We define an equivalence relation $\sim$ on $T$ by
 \[(\alpha,\beta) \sim (\alpha',\beta') :\Leftrightarrow \alpha+\beta=\alpha'+\beta' \textrm{ and }(\alpha_0,\beta_0)
 =(\alpha_0',\beta_0').\]
  Let $(\alpha,\beta) , (\alpha',\beta') \in T$. 
  We will show that if $(\alpha,\beta) \sim (\alpha',\beta')$  we have 
 \begin{equation}{\label{eqn:rumtauschen}}
  \underline{\x}^{\alpha}  \wedge \underline{\x}^{\beta} = 
  \underline{\x}^{\alpha'}  \wedge \underline{\x}^{\beta'} 
 \end{equation}
as elements in $\bigwedge^2 M_{\geq e}$.
We show this by induction on $r=|\beta-\beta'|$.
 If $r=0$, then we are done, otherwise there are $1 \leq i,j \leq n$ such that $\beta_i>\beta_i'$ and $\beta_j<\beta_j'$.
 Letting $\delta_i$ resp. $\delta_j$ be the $i$th resp. $j$th unit vector, we have
\[\underline{\x}^{\alpha}  \wedge \underline{\x}^{\beta} = x_j \cdot \underline{\x}^{\alpha-\delta_j}  \wedge \underline{\x}^{\beta}
=\underline{\x}^{\alpha-\delta_j}  \wedge \underline{\x}^{\beta+\delta_j}
=x_i \cdot \underline{\x}^{\alpha-\delta_j}  \wedge \underline{\x}^{\beta+\delta_j-\delta_i}
= \underline{\x}^{\alpha-\delta_j+\delta_i}  \wedge \underline{\x}^{\beta+\delta_j-\delta_i},\]
again as elements in $\bigwedge^2 M_{\geq e}$.
Since  $(\alpha-\delta_j+\delta_i,\beta+\delta_j-\delta_i) \sim (\alpha',\beta')$ and $|\beta+\delta_j-\delta_i-\beta'|<r$
we have by induction hypothesis
\[\underline{\x}^{\alpha}  \wedge \underline{\x}^{\beta} =\underline{\x}^{\alpha-\delta_j+\delta_i}  \wedge \underline{\x}^{\beta+\delta_j-\delta_i}
  =\underline{\x}^{\alpha'}  \wedge \underline{\x}^{\beta'} .\]
  With a similar argument one can show that $\underline{\x}^{\alpha}  \wedge \underline{\x}^{\beta}=0$ 
  (as elements in $\bigwedge^2 M_{\geq e}$) if $\alpha_0=\beta_0$ and $|\alpha|\geq e+1$, $|\beta|\geq e$.
  Thus, the elements of the form
 $\underline{\x}^{\alpha}  \wedge \underline{\x}^{\beta}$ where $(\alpha,\beta) \in T$ generate
 $\mathfrak{m}(\bigwedge^2 M_{\geq e})$ as a $\R$-vector space.
  Let $T' \subseteq T$ be a complete system of representatives of $\sim$.
 The identity $\ref{eqn:rumtauschen}$ implies then that the elements of the form
 $\underline{\x}^{\alpha}  \wedge \underline{\x}^{\beta}$ where $(\alpha,\beta) \in T'$ generate
 $\mathfrak{m}(\bigwedge^2 M_{\geq e})$ as a $\R$-vector space.
 Since those are mapped by $\wedge i$ to a set of $\R$-linear independent elements of $\bigwedge^2 M$ this implies  the Lemma.
\end{proof}

\begin{Thm}{\label{thm:constructdetrep}}
 If there exists a nice B\'ezoutian $\omega \in M \otimes_R M$, then there are symmetric matrices 
 $A_1, \ldots, A_n \in \Sym_N (\R)$ such that $h$
 divides the polynomial \[\det(x_0 \textnormal{I}_N -( x_1 A_1+ \ldots + x_n A_n)).\]
\end{Thm}

\begin{proof}
 By definition we can write
 \[\omega=v_1 \otimes v_1 + \ldots + v_r \otimes v_r,\]
 where $v_1, \ldots, v_r \in M_e$ generate $M_{\geq e}$ as an $R$-module.
 Let $F$ be the free graded $R$-module $R^r$ and let $e_1, \ldots, e_r$ be a basis of $F$.
 Consider the following $R$-linear map
\[m: F \to M, \,\, e_i \mapsto v_i.\]
Note that $m$ is homogeneous of degree $e$.
Consider the tensor \[\varphi= e_1 \otimes e_1 + \ldots + e_r \otimes e_r \in F \otimes_R F .\]
The associated bilinear form $\langle - , - \rangle_{\varphi}$ on $F^*$ is symmetric and admits an orthonormal basis, namely the
dual basis of $e_1, \ldots, e_r$.
Now consider the $R$-linear map \[f: M \to M, \, z \mapsto \overline{x}_0 \cdot z.\]
Since $f(M_{\geq e}) \subseteq M_{\geq e}$ there is an $R$-linear map $g: F \to F$ such that $f \circ m = m \circ g$.
Since $f$ is homogeneous of degree $1$ we can choose $g$ also to be homogeneous of degree $1$.
\[\begin{tikzcd}
   F \arrow{r}{g} \arrow{d}{m} & F \arrow{d}{m}\\
   M \arrow{r}{f} & M
  \end{tikzcd} \,\,\,\,\,\,\,\,\, \begin{tikzcd}
   F^* \arrow{r}{g^*}  & F^* \\
   M^* \arrow{r}{f^*}\arrow{u}{m^*} & M^* \arrow{u}{m^*}
  \end{tikzcd}\]
  The characteristic polynomial of $f$ resp. $f^*$ is $h$ and since $m$ is generically surjective $h$ divides the
  characteristic
  polynomial of $g$ resp. $g^*$. Therefore, if $g^*$ is selfadjoint with respect to $\langle - , - \rangle_{\varphi}$,
  we are done. Because
  in that case, a representing matrix of $g^*$ with respect to the orthonormal basis of $F^*$
  from above
  would be symmetric.
  Thus we consider the skew-symmetric bilinear form $\alpha=\varphi \circ g^* - g \circ \varphi$  and we want to choose
  $g$ in such a way that $\alpha=0$.
  Since $\omega$ is a B\'ezoutian, we have $\omega \circ f^* = f \circ \omega$ which, together with 
  $f \circ m = m \circ g$ and $\omega=m \circ \varphi \circ m^*$, implies that $m \circ \alpha \circ m^*=0$.
  By Remark \ref{rem:skewhom}, this corresponds to $\wedge m (\alpha)=0$ if we consider $\alpha$ as an
  element of $\bigwedge^2 F$ and if  $\wedge m: \bigwedge^2 F \to \bigwedge^2 M$ is the map induced by $m$.
  We can write $m= i \circ \tilde{m}$ where $\tilde{m}: F \to M_{\geq e}$ is surjective
  and $i : M_{\geq e} \hookrightarrow M$ is the inclusion map. By the preceding lemma $\wedge m(\alpha)=0$
  implies $\wedge \tilde{m} ( \alpha)=0$ since $\alpha$
  is homogeneous of degree $1$. Since $\tilde{m}$ is surjective $\alpha$ lies in the 
  submodule of $\bigwedge^2 F$ which is generated by elements of the form $v \wedge w$ where $v \in \ker(m)$
  and $w \in F$ (see for example \cite[Chapter III, §7, no. 2, Proposition 3]{bourbaki}).
  By Remark \ref{rem:skewkern} we can thus write $\alpha=\beta - \beta^*$ where $\beta: F^* \to F$ satisfies
  $\textrm{im}(m^*) \subseteq \ker(\beta)$.
  Now let $\tilde{g}=g-\beta^* \circ \varphi^{-1}$. We still have $m \circ \tilde{g}=f \circ m$ since 
   $\textrm{im}(m^*) \subseteq \ker(\beta)$, but $\tilde{g}^*$ is selfadjoint with respect to $\varphi$
   since $\alpha=\beta - \beta^*$.
   Therefore, a representing matrix  of $\tilde{g}$ with respect
   to the orthonormal basis of $\varphi$ is a  symmetric matrix and its characteristic
   polynomial is divisible by $h$.
   If $\tilde{g}$ is not homogeneous of degree one, we can replace it by its homogeneous part of degree one: 
   Every other homogeneous component of $\tilde{g}$ lies in the kernel of $m$, because $m$ and $f$ are homogeneous
   of degree $e$ and one.
\end{proof}
In the next section we will prove that  such a
nice B\'ezoutian indeed exists when our hyperbolic polynomial $h$ has no real singularities.

\section{A Positivstellensatz}{\label{sec:stellensatz}}
Let $R=\R[x_1, \ldots, x_n]$ be the polynomial ring equipped with the standard
grading. In this section
let $M$ always denote a free graded $R$-module.
\begin{Def}
  Let $\omega \in M \otimes_R M$ be symmetric. 
 We say that $\omega$ is \textit{a sum of squares} if \[\omega=v_1 \otimes v_1 + \ldots + v_r \otimes v_r\]
 for some $v_i \in M$.
\end{Def}

In what follows let $\omega \in M \otimes_R M$ always be symmetric.

\begin{Example}
 In the case $M=R$ we have the natural isomorphism $R=R \otimes_R R$. Under this isomorphism the sum of squares
 elements in $R \otimes_R R$ correspond to the sum of squares in $R$ in the usual sense.
 We denote the set of sum of squares in $R$ by $\sum R^2$.
\end{Example}

\begin{Remark}
 Every $\omega \in M \otimes_R M$ which is a sum of squares is also positive semidefinite.
\end{Remark}

\begin{Remark}
 Let $\omega \in M \otimes_R M$ be symmetric. We can consider $\omega$ as an element of $\Hom_R(M^*,M)$.
 Let $A$ be a representing matrix with respect to some basis of $M$ and its dual basis. Then $A$ is symmetric
 and has entries in $R$. Now $\omega$ is a sum of squares if and only if we can write 
 \[A=v_1 \cdot v_1^{\textrm{T}}+ \ldots v_r \cdot v_r^{\textrm{T}}\] for some vectors $v_1, \ldots, v_r$ of
 suitable size with entries in $R$, or equivalently $A=B \cdot B^{\textrm{T}}$ for some matrix $B$ of 
 suitable size with entries in $R$. In that case we say that the matrix $A$ is \textit{a sum of squares}.
\end{Remark}

In \cite{GonRib} the authors proved a generalization of Artin’s solution to Hilbert’s 17th problem. The following Lemma is a
slightly varied version of it.

\begin{Lemma}{\label{lem:posdefcofakt}}
 Let $\omega \in M \otimes_R M$ be homogeneous of degree $2 r$ and positive definite. Then there is a homogeneous
 $0 \neq q \in \sum R^2$ which is positive definite, such that $q\cdot \omega$ is a sum of squares.
\end{Lemma}

\begin{proof}
 We can consider $\omega$ as an element of $\Hom_R(M^*,M)$.
 Let $A$ be a representing matrix with respect to some basis of $M$ consisting of homogeneous elements
 $v_1, \ldots, v_m \in M$
 and its dual basis.
 Let $v_i$ be homogeneous of degree $r_i$ and let $\alpha_i=r-r_i$ for $1 \leq i \leq m$.
 Then $A$ is symmetric and the $(i,j)$-th entry of $A$ is homogeneous of degree $\alpha_i+\alpha_j$.
 Furthermore, since $\omega$ is positive definite, the matrix $A(p)$ obtained by plugging in
 $p=(p_1, \ldots, p_n)$ for $(x_1,\ldots, x_n)$ is positive definite for all $0 \neq p \in \R^n$.
 Let $\tilde{\ul{x}}=(x_1, \ldots, x_{n-1})$ and let
 $\tilde{A}=A(\tilde{\ul{x}},1)$
 be the matrix obtained by plugging in $1$ for $x_n$.
 We improve the proof in \cite{HilNie}.
 Let \[\vt\nolimits^m-a_{1} \cdot  \vt\nolimits^{m-1}+  \ldots + (-1)^m \cdot  a_m \in \R[\tilde{\ul{x}},\vt]\]
 be the characteristic polynomial of $\tilde{A}$. Note that $a_i$ is the sum of all symmetric $i \times i$ minors of $\tilde{A}$.
 Without loss of generality, we assume that $m$ is odd. If $m$ is even, then the argument is the same.
 By the Cayley–Hamilton theorem we have
 \[(\tilde{A}^{m-1} + a_{2}    \tilde{A}^{m-3}+ \ldots +a_{m-1}   I)   \tilde{A}= a_1  \tilde{ A}^{m-1} + a_{m-3}   \tilde{A}^{m-3}+ \ldots + a_m  I.\]
 Let $B=\tilde{A}^{m-1} + a_{2}    \tilde{A}^{m-3}+ \ldots +a_{m-1}   I$. By Stengle's Positivstellensatz (see \cite[Corollary 4.4.3]{BCR})
 there exists a polynomial $q \in \R[\tilde{\ul{x}}]$ which is strictly positive on $\R^{n-1}$
 and a sum of squares such that $q \cdot  a_i$
 is a sum of squares for all $1 \leq i \leq m$, since every symmetric minor of $\tilde{A}$ is strictly positive on $\R^{n-1}$.
 Therefore $q  B$ is a sum of squares. From the elementary properties of the adjugate matrix we obtain
 \[q^{2}   \det(B)^2   \tilde{A}= (q   B) \cdot  \adj(B)^2 \cdot (q a_1   \tilde{A}^{m-1}  + q a_{m-3}   \tilde{A}^{m-3}+ \ldots + q a_m  I).\]
 Let $f=q^{2}   \det(B)^2$ and let  $2 d_n=\deg(f)$. For all $ v \in \R^{n-1}$ we have $f(v)>0$.
 Because $q B$ is a sum of squares and $B$ and $\adj(B)$ commute with $\tilde{A}$, it follows that 
 $f   \tilde{A}$ is a sum of squares, i.e. 
 we have $f \tilde{A}=S^{\rm T} S$ for some matrix $S=(s_{ij})_{1\leq i \leq m', 1\leq j \leq m}$ with $s_{ij} \in \R[\tilde{\ul{x}}]$.
 Since $A(p)$ is positive definite for all $0 \neq p \in \R^n$,
 the $j$th diagonal entry of $\tilde{A}$ has degree $2 \alpha_j$.
 Thus we have $\max_{i}(\deg(s_{ij}))=\alpha_j+d_n$.
 Let \[p_n=x_n^{2 d_n} f(\frac{x_1}{x_n}, \ldots, \frac{x_{n-1}}{x_n}).\]
 The polynomial $p_n$ is homogeneous of degree $2d_n$ and for all $v \in \R^n$ we have $p_n(v)>0$ if $v_n\neq 0$.
 It is easy to see that we have $p_n A=S'^{\rm T} S'$ where $S'=(s'_{ij})_{ij}$ and
 \[s'_{ij}=x_n^{\alpha_j+d_n} s_{ij}(\frac{x_1}{x_n}, \ldots, \frac{x_{n-1}}{x_n}).\]
 
 Replacing $x_n$ by $x_i$ for $1 \leq i \leq n$ we obtain a homogeneous polynomial $p_i \in \R[\ul{x}]_{2d_i}$ such that
 $p_i(v)>0$ for all $v \in \R^n$ with $v_i \neq 0$ and such that $p_i A$ is a sum of squares.
 Let $d=\max_i(d_i)$ and consider the homogeneous polynomial
 \[p=x_1^{2(d-d_1)} p_1+ \ldots + x_n^{2(d-d_n)} p_n.\]
 It follows that $p$ is positive definite and that $pA$ is a sum of squares.
\end{proof}

\begin{Def}
 Let $\omega \in M \otimes_R M$ be homogeneous. We denote by $\Sigma(\omega)$ the submodule of $M$ which
 is generated by all homogeneous elements $v \in M$ such that there are homogeneous $q \in \sum R^2$
 and $w_1, \ldots, w_r \in M$ with
 \[q \cdot \omega = v \otimes v + \sum_{i=1}^r w_i \otimes w_i.
 \]
\end{Def}

\begin{Example}{\label{exp:innerpoint}}
 Let $p=(x_1^2+\ldots +x_n^2) \in R$. We will show by induction on $r$ that $\Sigma(p^r)=R_{\geq r}$.
 The case $r=0$ is clear. Let $r>0$ and assume that the claim is true for all $0 \leq r' <r$.
 It is easy to see that $\Sigma(p^r)\subseteq R_{\geq r}$. The identity
 \[q p^{r-1} =  \sum_{j=1}^s g_j^2 \]
 for homogeneous polynomials $q, g_1, \ldots , g_s \in R$ implies 
 \[q p^r = \sum_{i=1}^n \sum_{j=1}^s (x_i g_j)^2 .\]
 Therefore we have $R_{\geq r}=(x_1, \ldots, x_n) \cdot \Sigma(p^{r-1}) \subseteq \Sigma(p^r)$.
\end{Example}

\begin{Remark}{\label{rem:addingsos}}
 It follows immediately from the definition that we always have
 \[\Sigma(\omega) \subseteq \Sigma(\omega+v \otimes v)\]
 where $\omega \in M \otimes_R M$ is homogeneous of degree $2r$ and $v \in M$ homogeneous of degree $r$.
\end{Remark}

\begin{Lemma}
 Let $\omega \in M \otimes_R M$ be homogeneous and let $v_1, \ldots, v_r \in \Sigma(\omega)$ be homogeneous of
 the same degree. Then there are homogeneous $q \in \sum R^2$ and $w_1, \ldots, w_s \in M$ such that
 \[
  q \cdot \omega = \sum_{i=1}^r v_i \otimes v_i + \sum_{j=1}^s w_j \otimes w_j.
 \]
\end{Lemma}

\begin{proof}
 This is straightforward if we use the identity 
 \[v_1 \otimes v_1 + v_2 \otimes v_2 = \frac{1}{2} (v_1+v_2) \otimes (v_1+v_2)+\frac{1}{2} (v_1-v_2) \otimes (v_1-v_2).\]
\end{proof}

\begin{Lemma}
 We consider elements $\omega \in M \otimes_R M$, $v \in M$, $f \in R$ and  $g \in \Sigma(f) \subseteq R$, all of them 
 homogeneous. 
 If $f\cdot v \in \Sigma(\omega)$, then $g^2\cdot v \in \Sigma(\omega)$.
\end{Lemma}

\begin{proof}
 By the previous lemma there  are homogeneous $q_1, q_2 \in \sum R^2$, $w_1, \ldots, w_r \in M$
 and $h_1, \ldots , h_s \in R$
 such that
 \[q_1 \omega = f^2\cdot (v \otimes v) + \sum_{i=1}^r w_i \otimes w_i \,\,
\textrm{ and  }\,\,
 q_2 f=g^2+ \sum_{j=1}^s h_i^2.
\]
If we multiply the first equation with $q_2^2$ we obtain
\begin{eqnarray*}
  q_2^2 q\cdot \omega &=& (g^2+ \sum_{j=1}^s h_i^2)^2 \cdot (v \otimes v) + \sum_{i=1}^r q_2w_i \otimes q_2w_i \\
& =& (g^2 v \otimes g^2 v)+ ((\sum_{j=1}^s h_i^2)^2+2 \cdot \sum_{j=1}^s (g h_i)^2) v \otimes v + \sum_{i=1}^r q_2w_i \otimes q_2w_i.
\end{eqnarray*}
\end{proof}

\begin{Cor}{\label{cor:opaufsigma}}
 Let $\omega \in M \otimes_R M$ and $f \in (\Sigma(\omega):M)$ be both
 homogeneous. Then $g^2 \in (\Sigma(\omega):M)$ for all homogeneous $g \in \Sigma(f)$.
\end{Cor}

\begin{Lemma}{\label{lem:norealproj}}
 If $\omega \in M \otimes_R M$ is homogeneous and positive definite,
 then the projective zero set of the homogeneous ideal $(\Sigma(\omega) : M)=\{a \in R: \, aM \subseteq \Sigma(\omega)\}$
 contains no real points.
\end{Lemma}

\begin{proof}
 By Lemma \ref{lem:posdefcofakt} there is a homogeneous, positive definite $q \in \sum R^2$ such that 
 \[q \omega=w_1 \otimes w_1 + \ldots + w_r \otimes w_r \] for some $w_1, \ldots, w_r \in \Sigma(\omega)$.
 Let $v_1, \ldots, v_m$ be a basis of $M$ consisting of homogeneous elements. 
 We write $w_j=g_{1j} \cdot v_1+\ldots+ g_{mj}\cdot v_m$ with $g_{ij} \in R$ for all $1 \leq j \leq r$.
 Consider the matrix $G=(g_{ij})_{1 \leq i \leq m, \, 1 \leq j \leq r}$.
 Let $S \subseteq \{1, \ldots ,r\}$ be a subset with exactly $m$ elements. Let
 $G_{S}=(g_{i j})_{1 \leq i \leq m,j \in S}$ be the corresponding submatrix of $G$ and 
 $p_{S}=\det(G_{S})$ its determinant. Since the representing matrix of $q \omega \in \Hom_R(M^*,M)$
 with respect to the basis $v_1, \ldots, v_m$ and its dual basis is $G \cdot G^{\textrm{T}}$ and since
 $q \omega$ is positive definite, the polynomials in 
 the set \[\{p_S: \, S \subseteq \{1, \ldots ,r\}, \,\, |S|=m\}\] have no common projective real zero. From
 $G_S \cdot \adj(G_S)=p_S \cdot \textrm{I}_m$, where $\textrm{I}_m$ is the identity matrix,
 we immediately obtain $p_S\cdot v \in  \Sigma(\omega)$ for all $v \in M$.
\end{proof}

For the next step, we will need the following Real Nullstellensatz, cf. \cite[Corollary 4.1.9.]{BCR}.
\begin{Thm}
 Let $V$ be an affine real variety and let $V(\R)$ be the set of its real points.
 Let $I$ be an ideal of the coordinate ring $\R[V]$ and let \[Z=\{x \in V(\R): \,\, f(x)=0 \textnormal{ for all } f \in I\}\]
 be its real zero set.
  For all $p \in \R[V]$  the following are equivalent:
 \begin{enumerate}[(i)]
  \item $p(x)=0$ for all $x \in Z$.
  \item There are $g_1, \ldots, g_r \in \R[V]$ such that $p^{2m}+g_1^2+\ldots+g_r^2 \in I$ for some $m \geq 0$.
 \end{enumerate}
\end{Thm}

\begin{Cor}
 Let $I$ be a homogeneous ideal  in $R$ whose projective zero set contains no real point.
 Then we can find a homogeneous $f \in I$ which satisfies $\Sigma(f)=R_{\geq r}$ for some suitable $r\geq 0$.
\end{Cor}

\begin{proof}
 Consider the projective space $\mathbb{P}^{n-1}=\mathbb{P}(\C^{n})$. 
 Let $q=x_1^2+\ldots + x_n^2$ and consider $Z=\{x \in \mathbb{P}^{n-1}: \, q(x)=0\}$.
 Then $V=\mathbb{P}^{n-1} \smallsetminus Z$ is an affine $\R$-variety with coordinate ring
 \[\R[V]=\{\frac{p}{q^k}: \,\, p \in R \textnormal{ homogeneous, } \deg(p)=\deg(q^k)\}.\]
 We consider the following ideal in $\R[V]$:
 \[J=\{ \frac{p}{q^k}: \,\, p\in I \textnormal{ homogeneous, } \deg(p)=\deg(q^k)\}.\]
 The real zero set of $J$ is empty by assumption. Thus there are, by the preceding Theorem,
 $g_1, \ldots, g_s \in \R[V]$ such that \[1+g_1^2+\ldots+g_s^2 \in J.\]
 After cleaning denominators, Example \ref{exp:innerpoint} and Remark \ref{rem:addingsos} imply the claim.
\end{proof}

\begin{Thm}
 Let $\omega \in M \otimes_R M$ be homogeneous and positive definite. Then we have 
 $R_{\geq k} \subseteq (\Sigma(\omega):M)$ for some $k \geq 0$.
\end{Thm}

\begin{proof}
 The projective zero set of the homogeneous ideal $(\Sigma(\omega) : M)$
 contains no real points by Lemma \ref{lem:norealproj}. 
 By the preceding corollary there is a homogeneous $f \in (\Sigma(\omega) : M)$ such that
 $\Sigma(f)=R_{\geq r}$ for some suitable $r\geq 0$. Thus for all $p \in R_{\geq r}$ we have 
 $p^2 \in (\Sigma(\omega) : M)$ by Corollary \ref{cor:opaufsigma}.
 This implies that the projective zero set of $(\Sigma(\omega) : M)$ is empty.
 It follows from Hilbert's Nullstellensatz that $R_{\geq k} \subseteq (\Sigma(\omega):M)$ for some $k \geq 0$.
\end{proof}

\begin{Cor}
 Let $\omega \in M \otimes_R M$ be homogeneous and positive definite.
 Then $M_{\geq k} \subseteq \Sigma(\omega)$ for some $k \geq 0$.
\end{Cor}

\begin{Cor}{\label{cor:stellensatz}}
 Let $\omega \in M \otimes_R M$ be homogeneous and positive definite.
 Then there is some $q \in \sum R^2$ such that
 \[q \cdot \omega = v_1 \otimes v_1 + \ldots + v_r \otimes v_r \]
 where $v_1, \ldots , v_r \in M_k$ generate $M_{\geq k}$ for some $k \geq0$.
\end{Cor}

Now we are ready to proof our main result.
\begin{Thm}{\label{thm:main}}
 Let $h \in \R[x_0, \ldots, x_n]$
 be hyperbolic with respect to $e \in \R^{n+1}$. Assume that $h$ has no real
 singularities (i.e. $\nabla h (v) \neq 0$ for all
 $0 \neq v \in \R^{n+1}$). Then there is a hyperbolic polynomial 
 $q \in \R[x_0, \ldots, x_n]$,  such that $q\cdot h$ has a definite determinantal representation.
\end{Thm}

\begin{proof}
 After a linear change of coordinates we can assume that $e=(1,0,\ldots,0)^{\textrm{T}}$.
 Let $S=\R[x_0,\ldots,x_n]$ and $M=S/(h)$. By Theorem \ref{thm:bezhyp}
 there is a homogeneous B\'ezoutian  $\omega \in M \otimes_R M$ which is positive definite.  
 By Corollary \ref{cor:stellensatz} 
 there is some $q \in \sum R^2$ such that
 \[q \cdot \omega = v_1 \otimes v_1 + \ldots + v_r \otimes v_r \]
 where $v_1, \ldots , v_r \in M_k$ generate $M_{\geq k}$ for some $k \geq0$.
 It is easy to see that $\omega'=q \cdot \omega$ is again a B\'ezoutian. In particular $\omega'$ is
 a nice B\'ezoutian. Now Theorem \ref{thm:constructdetrep} implies the claim.
\end{proof}

\noindent \textbf{Acknowledgements.}
This work is part of my PhD thesis. I would like
to thank my advisor Claus Scheiderer for his encouragement and the
Studienstiftung des deutschen Volkes for their financial and ideal support.
I also thank Christoph
Hanselka, Tim Netzer, Daniel Plaumann, Eli Shamovich, Bernd Sturmfels, Andreas Thom and Cynthia Vinzant
for helpful discussions.

\end{document}